\documentclass[12pt,final]{amsart}

\pdfoutput=1
\usepackage{amsmath,amsthm,amssymb}
\usepackage{amsfonts}
\usepackage[whole]{bxcjkjatype}
\usepackage[top=25truemm,bottom=25truemm,left=30truemm,right=30truemm]{geometry}
\usepackage{url}

\usepackage[varg]{newtxmath}
\usepackage{newtxtext}
\usepackage[mathscr]{eucal}
\usepackage{bxpapersize}
\usepackage{ascmac}
\usepackage{epic, eepic}
\usepackage{graphicx}
\usepackage{mathtools}
\usepackage{indentfirst}
\usepackage{cases}
\usepackage{color}
\usepackage{url}
\usepackage{bm}
\usepackage[all]{xy}
\usepackage[noadjust]{cite}
\usepackage{showkeys}
\usepackage{enumerate}
\usepackage{hyperref}
\hypersetup{
setpagesize=false,
 bookmarksnumbered=true,%
 bookmarksopen=true,%
 colorlinks=true,%
 linkcolor=blue,
 citecolor=blue,
 urlcolor=magenta,
}

\theoremstyle{plane}
\newtheorem{theorem}{\bf Theorem}[section]
\newtheorem{proposition}[theorem]{\bf Proposition}            
\newtheorem{corollary}[theorem]{\bf Corollary}                
\newtheorem{lemma}[theorem]{\bf Lemma}
\newtheorem{remark}[theorem]{\bf Remark}

\newtheorem{fact}[theorem]{\bf Fact}

\newtheorem{definition}[theorem]{\bf Definition}

\renewcommand{\phi}{\varphi}

\numberwithin{equation}{section}

\title[Classifications of cusps appearing on plane curves]
{Classifications of cusps appearing on plane curves}
\author[Y. Matsushita]{Yoshiki Matsushita}
\address{Graduate school of Mathematics, Kyushu University,  744 Motooka, Nishi-ku,  Fukuoka 819-0395, Japan}
\email{matsushita.yoshiki.297@s.kyushu-u.ac.jp}

\thanks{This work was supported by JST SPRING Grant Number JPMJSP2136.}
\date{\today}

\begin{document}

\begin{abstract}
In this paper, we deal with plane curves with cusps. It is well known that there are various types of cusps. Among them, we investigate criteria for $(n, n+1)$ cusps with respect to several differential conditions and relations between these singularities and evolutes of fronts. We give complete classifications with respect to $(4, 5)$-cusps.
\end{abstract}
\maketitle

\section{\bf Introduction}\label{sec:Introduction}
One of the most basic singularities that appear in plane curves are cusps. A plane curve $\gamma$ is $(n, m)$-cusp at $t=0$ if $\gamma$ is $\mathcal{A}$-equivalent to $Cusp_{(n,m)}(t) = (t^n, t^m)$ at $t=0$. $(n,m)$-cusps are known to have applications to optical systems (\cite{Pei and Zhang}). In general, it is difficult to examine the types of $(n,m)$-cusps in the sense of $\mathcal{A}$-equivalence. However, simple criteria for $(2, 3)$-cusps, $(2, 5)$-cusps, $(2, 7)$-cusps, $(3, 4)$-cusps and $(3, 5)$-cusps are known (\cite[Theorem 1.3.2]{SUY1}, \cite[Theorem 1.23]{Porteous},  \cite[Theorem A.1]{Honda}, \cite[Theorem 1.3.4]{SUY1} and \cite[Fact 2.1]{MSST02}). Namely,

\begin{fact}\label{fact:(2,3)cusp and (2,5)cusp and (2,7)cusp and (3,4)cusp and (3,5)cusp}
Let $\gamma : I \subset \mathbb{R} \rightarrow \mathbb{R}^2$ be a $C^{\infty}$ curve.  
\begin{itemize}
\item[(1)] $\gamma$ is $\mathcal{A}$-equivalent to $Cusp_{(2, 3)}(t) = (t^{2}, t^{3})$ at $t = 0$ if and only if $\gamma'(0) = \bold{0}$ and $\det{(\gamma''(0), \gamma'''(0))} \neq 0.$

\item[(2)] $\gamma$ is $\mathcal{A}$-$equivalent$ to $Cusp_{(2, 5)}(t) = (t^2, t^5)$ at $t = 0$ if and only if $\gamma'(0) = \bold{0}$, $\det{(\gamma''(0), \gamma'''(0))} = 0$ and 

\begin{equation*}
3 \det{(\gamma''(0), \gamma^{(5)}(0) )} \gamma''(0) - 10 \det{(\gamma''(0), \gamma^{(4)}(0))} \gamma'''(0) \neq \bold{0}. 
\end{equation*}

\item[(3)] $\gamma$ is $\mathcal{A}$-$equivalent$ to $Cusp_{(2, 7)}(t) = (t^2, t^7)$ at $t = 0$ if and only if there exist real numbers $k, l \in \mathbb{R}$ such that $\gamma'(0) = \bold{0}$, $\gamma''(0) \neq \bold{0}$, $\gamma'''(0) = k \gamma''(0)$, $\gamma^{(5)}(0) - \dfrac{10}{3} k \gamma^{(4)} (0) = \ell \gamma''(0)$ and

\begin{align*}
\det{\Bigl( \gamma''(0), \gamma^{(7)}(0)  - 7k \gamma^{(6)} (0) - \Bigl(  7 \ell - \dfrac{70}{3} k^3 \Bigr) \gamma^{(4)}(0) \Bigr) } \neq 0.
\end{align*}

\item[(4)] $\gamma$ is $\mathcal{A}$-equivalent to $Cusp_{(3, 4)}(t) = (t^{3}, t^{4})$ at $t = 0$ if and only if $\gamma'(0) = \gamma''(0)= \bold{0}$ and $\det{(\gamma'''(0), \gamma^{(4)}(0))} \neq 0$.

\item[(5)] $\gamma$ is $\mathcal{A}$-equivalent to $Cusp_{(3, 5)}(t) = (t^{3}, t^{5})$ at $t = 0$ if and only if $\gamma'(0) = \gamma''(0) = \bold{0}$, $\det{(\gamma'''(0), \gamma^{(4)}(0))} = 0$ and $\det{(\gamma'''(0), \gamma^{(5)}(0))} \neq 0$. 

\end{itemize}
\end{fact}
It is known that $(2,3)$-cusp appears on cycloids and asteroids and $(3,4)$-cusp appears on parallel curves of parabolas (cf.\cite{SUY1}). Moreover, $(2, 3)$-cusp is the generic singularity of fronts. Looking at the criteria for $(2, 3)$-cusps and $(3, 4)$-cusps stated in Fact \ref{fact:(2,3)cusp and (2,5)cusp and (2,7)cusp and (3,4)cusp and (3,5)cusp}, it seems that they can be naturally extended to criteria for $(n, n+1)$ cusps. However, this does not work in the case of $(4, 5)$-cusps. For example, $(t^4, t^5 + t^7)$ , $(t^4, t^5 - t^7)$ and $(t^4, t^5)$ are not $\mathcal{A}$-equivalent at $t=0$. Namely, situations of $(4, 5)$-cusps are different from the cusps whose criteria have already been obtained, and are more complicated because they require careful treatments of information of larger orders. Moreover, we are also interested in how far $(n,n+1)$-cusps can be classified by the naturally extended conditions just mentioned. 

On the other hand, $(n, n+1)$-cusps are wave fronts. Fukunaga and Takahashi showed that if we consider an evolute of a wave front which does not have inflection points, then a singular point of the initial front corresponds to a singularity of an evolute (\cite[Proposition $5.2$ $(1)$ and $(3)$]{Fukunaga and Takahashi}). 
By using this fact and  invariants of $(n, n+1)$-cusps, if we perform the operation of considering the $(n-1)$-th evolute for a plane curve with a $(n, n+1)$ cusp, then it becomes a regular curve. 
Hence, $(n,n+1)$-cusps can be recognized as singularities which become regular points when we consider the $(n-1)$-th evolutes of the fronts (Theorem \ref{coro:properties of (n,n+1)cusps}). 

In this paper, we newly introduce an invariant $\kappa_{q}$ of $(4, 5)$-cusp other than the $(4, 5)$-cuspidal curvature (Definition \ref{definition:new (4,5) cusipdal}). 
For the {\it $(n, n+1)$-cuspidal curvature}, see (\cite[Appendix A]{MSST01}). This new invariant $\kappa_{q}$ is called the {\it $(4, 5; \pm 7)$-cuspidal curvature}. 
In Section \ref{sec:Classification of $(n, n+1)$ cusp in $C^{1}$ class diffeomorphism}, we give criteria for $(n, n+1)$-cusps in the sense of $C^{1}$-equivalence  (Theorem \ref{thm:C1-diffeo-classification}). 
This corresponds to the classifications which can be done by the natural extension mentioned at the beginning of this paper. 
Moreover, we prove properties of relations between $(n, n+1)$-cusps and evolutes of fronts (Theorem \ref{coro:properties of (n,n+1)cusps}). 
Furthermore, we give negative criteria for $(n,n+1)$-cusps in the sense of $\mathcal{A}$-equivalence by using properties of evolutes of fronts and invariants of $(n,n+1)$-cusps (Corollary \ref{coro:negative criterion for (n,n+1)-cusps}).
In Section \ref{sec:criterion for (4, 5)cusp}, we give criterion for $(4, 5)$-cusps (Theorem \ref{thm:criterion for (4, 5)cusp}). 
Considering the $(4, 5; \pm 7)$-cuspidal curvature $\kappa_{q}$, we obtain the complete classifications of $(t^4, t^5 + t^7)$ , $(t^4, t^5 - t^7)$ and $(4, 5)$-cusps (Theorem \ref{thm: complete classifications of (4, 5)cusps }).

\section{\bf Preliminaries}
We recall some notions which we need in the following sections. 

\subsection{Fronts and singularities}Let $I$ be an open interval containing the origin of $\mathbb{R}$ and $\gamma : I \rightarrow \mathbb{R}^2$ be a $C^{\infty}$ map. A point $p$ is called a \textit{singular point} of $\gamma$ if $\gamma'(p) = \bold{0}$. We say that the curve $\gamma$ is a \textit{front} if there exists a $C^{\infty}$ map $\nu : I \rightarrow \mathbb{R}^2$ such that the pair $(\gamma, \nu) : I \rightarrow \mathbb{R}^2 \times \it{S}^1$ is a \textit{Legendre immersion}, namely, $\gamma'(t) \cdot \nu(t) = 0$ for each $t \in I$ and $(\gamma, \nu)$ is an immersion. For properties of fronts, see \cite{Fukunaga and Takahashi, Fukunaga and Takahashi 2, SUY1, SUY2}.

\begin{definition}\label{def:A-equivalent}
Let $\gamma_{1}, \gamma_{2} : (\mathbb{R}, 0) \rightarrow (\mathbb{R}^2, \bold{0})$ be $C^{\infty}$ map germs. We say that $\gamma_{1}$ and $\gamma_{2}$ are $C^{r}$-$equivalent$ if there exist $C^{r}$ diffeomorphism germs $\psi : (\mathbb{R}, 0) \rightarrow (\mathbb{R}, 0)$ and $\Psi : (\mathbb{R}^2, \bold{0}) \rightarrow (\mathbb{R}^2, \bold{0})$ such that $\Psi \circ \gamma_{1} \circ \psi^{-1} = \gamma_{2}$ holds, where $r \in \mathbb{N}_{> 0} \cup \{ \infty, \omega \}$. In particular, if $r = \infty$, we call it $\mathcal{A}$-$equivalent$.
\end{definition}

In Section\ref{sec:Introduction}, we introduced criteria for several cusps. Moreover, the following facts are known. 

\begin{fact}(\cite{Ishikawa})\label{fact:criterion for (2,n) cusp}
Let $\gamma : I \rightarrow \mathbb{R}^2$ be a $C^{\infty}$ curve. If  $\gamma$ satisfies the conditions $\gamma'(0) = \bold{0},\gamma''(0) \neq \bold{0}, \gamma'''(0) = \cdots = \gamma^{(n-1)}(0) = \bold{0}$ and 

\begin{align}
\det{(\gamma''(0), \gamma^{(n)}(0))} \neq 0, 
\end{align}
then $\gamma$ is $\mathcal{A}$-$equivalent$ to $Cusp_{(2, n)}(t) = (t^2, t^n)$ at $t = 0$, where $n$ is an odd number greater than or equal to $9$.
\end{fact}

Fact \ref{fact:(2,3)cusp and (2,5)cusp and (2,7)cusp and (3,4)cusp and (3,5)cusp} and Fact \ref{fact:criterion for (2,n) cusp} relate to the following facts which were proved by Bruce and Gaffney (\cite[Theorem 3.8]{Bruce and Gaffney}).

\begin{fact}\label{fact:simple germs}
The following are representatives of the $C^{\omega}$-$equivalent$ simple germs $(\mathbb{C}, 0) \rightarrow (\mathbb{C}^2, \bold{0}) :$
\begin{eqnarray}
(1) &A_{2k}& : t \rightarrow (t^2, t^{2k+1}), \nonumber \\
(2) &E_{6k}& : t \rightarrow (t^3, t^{3k+1} + t^{3k+p+2}); 0 \leq p \leq k-2, \quad t \rightarrow (t^3, t^{3k+1}), \nonumber \\
(3) &E_{6k+2}& : t \rightarrow (t^3, t^{3k+2} + t^{3k+p+4}); 0 \leq p \leq k-2, \quad t \rightarrow (t^3, t^{3k+2}), \nonumber \\
(4) &W_{12}& : t \rightarrow (t^4, t^5 + t^7), \quad t \rightarrow (t^4, t^5), \nonumber \\
&W_{18}& : t \rightarrow (t^4, t^7 + t^9), \quad t \rightarrow (t^4, t^7 + t^{13}), \quad t \rightarrow (t^4, t^7), \nonumber \\
(5) &W_{1, 2q-1} ^{\#}& : t \rightarrow (t^4, t^6 + t^{2q + 5}); q \geq 1. \nonumber
\end{eqnarray}
\end{fact}
For the definition of \textit{simple}, see \cite[Definition 2.6]{Bruce and Gaffney}.  As mentioned in Section \ref{sec:Introduction}, we discuss Fact \ref{fact:simple germs} $(4)$ in Section \ref{sec:criterion for (4, 5)cusp}.


\subsection{Evolutes of fronts}
Let $\gamma : I \rightarrow \mathbb{R}^2$ be a regular curve. We set $\bold{e}(t) = \dot{\gamma}(t)/ \| \dot{\gamma}(t) \|$ and $\bold{n}(t) = M(\bold{e}(t))$, where $\dot{\gamma} (t) = (d \gamma / dt) (t)$, $\| \dot{\gamma}(t) \| = \sqrt{\dot{\gamma}(t) \cdot \dot{\gamma}(t)}$ and $M$ is the anticlockwise rotation by $\pi/2$. Moreover, $\kappa(t) = \det{(\dot{\gamma}(t), \ddot{\gamma}(t))}/\| \dot{\gamma}(t) \|^3 = \dot{\bold{e}}(t) \cdot \bold{n}(t)/\| \dot{\gamma}(t) \|$ is called the \textit{curvature} of $\gamma$. An \textit{evolute} $Ev(\gamma) : I \rightarrow \mathbb{R}^2$ of $\gamma$ is given by 

\begin{align}
Ev(\gamma)(t) = \gamma(t) + \dfrac{1}{\kappa(t)} \bold{n}(t) 
\end{align}
except for the points where $\kappa(t) = 0$ (cf. \cite{Fukunaga and Takahashi, UY}). A point $t$ is called a \textit{inflection point} if $\kappa(t) = 0$. In general, if $\gamma$ has a singular point, then we cannot define the evolute as above, because $\kappa(t)$ may diverge at a singular point. However, Fukunaga and Takahashi defined an evolute of a front and gave the representation formula of it (\cite[Definition 2.10, Theorem  3.3]{Fukunaga and Takahashi}). Let us see this in the following. Let $(\gamma, \nu) : I \rightarrow \mathbb{R}^2 \times \it{S}^1$ be a Legendre immersion. We call the curve defined by $\gamma_{\lambda}(t) = \gamma(t) + \lambda \nu(t)$ the \textit{parallel curve} of $\gamma$, where $\lambda$ is a real number. Moreover, we denote the curvature of the parallel curve $\gamma_{\lambda}(t)$ by $\kappa_{\lambda}(t)$ when $\gamma_{\lambda}$ is a regular curve. 




\begin{definition}
Let $\gamma : I \rightarrow \mathbb{R}^2$ be a front without inflection points, that is, $\kappa(t) \neq 0$. Define an evolute $Ev(\gamma) : I \rightarrow \mathbb{R}^2$ of the front $\gamma$ by

\begin{subnumcases}{Ev(\gamma)(t) = }
\gamma(t) + \dfrac{1}{\kappa(t)} \bold{n}(t) \quad \text{if $t$ is a regular point of $\gamma$,} \nonumber \\
\gamma_{\lambda}(t) + \dfrac{1}{\kappa_{\lambda}(t)} \bold{n}_{\lambda}(t) \quad \text{if $t = t_{0}$ is a singular point of $\gamma$ }, \nonumber
\end{subnumcases}
where $\bold{n}_{\lambda}(t) = \dfrac{1 - \lambda \kappa(t)}{\vert 1 - \lambda \kappa(t) \vert} \bold{n}(t)$ and $\lambda \in \mathbb{R}$ is satisfied the condition $\lambda \neq 1/\kappa(t)$.

\end{definition}

Moreover, they proved that the following fact (\cite[Proposition 5.2 $(1)$ and $(3)$]{Fukunaga and Takahashi}).

\begin{fact}\label{fact:evolute of one singlarity down}
Let $\gamma : I \rightarrow \mathbb{R}^2$ be a front without inflection points. Assume that $t_{0}$ is a singular point of $\gamma$. Then $(d^i \gamma/dt^i )(t_{0}) = \bold{0}$ for $i = 2, \cdots, n+1$ if and only if $t_{0}$ is a singular point of $Ev^i(\gamma)$ for $i = 1, \cdots, n$, where $Ev^i(\gamma)$ is the $i$-th evolute of the front, that is, $Ev^i(\gamma)(t) = Ev(Ev^{i-1}(\gamma))(t)$. 
\end{fact}

By Fact \ref{fact:evolute of one singlarity down}, we immediately have the following result.

\begin{corollary}\label{coro:Evolute of regualr points}
Let $\gamma : I \rightarrow \mathbb{R}^2$ be a front without inflection points. Assume that $t_{0}$ is a singular point of $\gamma$. Then $(d^2 \gamma/dt^2) (t_{0}) = \bold{0}$, $\cdots$, $(d^{(n-1)} \gamma / dt^{(n-1)}) (t_{0}) = \bold{0}$ and $(d^n \gamma / dt^n ) (t_{0}) \neq \bold{0}$ if and only if $t_{0}$ is a singular point of $Ev(\gamma), \cdots Ev^{n-2}(\gamma)$ and a regular point of $Ev^{n-1}(\gamma)$.
\end{corollary}

As mentioned in Section \ref{sec:Introduction}, we give properties of $(n,n+1)$ cusps by using Fact \ref{fact:evolute of one singlarity down} and Corollary \ref{coro:Evolute of regualr points} in Section \ref{sec:Classification of $(n, n+1)$ cusp in $C^{1}$ class diffeomorphism}.


In order to give another representation of the evolute of a front, we recall the Frenet formula of a front. Let $(\gamma, \nu) : I \rightarrow \mathbb{R}^2 \times \it{S}^1$ be a Legendre immersion. We set $\bold{\mu}(t) = M(\nu(t))$ and  $\ell(t) = \nu'(t) \cdot \bold{\mu}(t)$. The pair $\{ \nu(t), \bold{\mu}(t) \}$ is called a \textit{moving frame along of $\gamma$}. Then the Frenet formula of $\gamma$ is given by

\begin{equation*}
\begin{pmatrix} \nu'(t) \\ \bold{\mu}'(t)  \end{pmatrix} =
\begin{pmatrix}
0 & \ell(t) \\ -\ell(t) & 0
\end{pmatrix}
\begin{pmatrix} \nu(t) \\ \bold{\mu}(t)  \end{pmatrix}.
\end{equation*}
We set $\beta(t) = \gamma'(t) \cdot \bold{\mu}(t)$ and we call the pair $(\ell, \beta)$ \textit{the curvature of the Legendre immersion.}  Note that $t_{0}$ is an inflection point of $\gamma$ if and only if $\ell(t_{0}) = 0$. By using the moving frame of a front, the evolute of a front can be written as follow (\cite[Theorem 3.3, Theorem 4.1]{Fukunaga and Takahashi}).

\begin{fact}
Let $(\gamma, \nu) : I \rightarrow \mathbb{R}^2 \times \it{S}^1$ be a Legendre immersion without inflection points. 
\begin{enumerate}
\item The evolute of a front $Ev(\gamma)$ can be expressed by
\begin{align*}
Ev(\gamma)(t) = \gamma(t) - \frac{\beta(t)}{\ell(t)} \nu(t).
\end{align*}
Moreover, $(Ev(\gamma), \mu)$ is a Legendre immersion with the curvature 
\begin{align*}
\left( \ell(t),  \frac{d}{dt} \left( \frac{\beta(t)}{\ell(t)} \right) \right).
\end{align*}
\item The evolute of an evolute of a front $Ev(Ev(\gamma))$ can be expressed by 
\begin{align*}
Ev(Ev(\gamma))(t) = Ev(\gamma)(t) - \frac{\beta^{\prime}(t) \ell(t) - \beta(t) \ell^{\prime}(t)}{\ell^3(t)} \mu(t).
\end{align*}
Moreover, $(Ev(Ev(\gamma)), -\nu)$ is a Legendre immersion with the curvature 
\begin{align*}
\left( \ell(t),  \frac{d}{dt} \left( \frac{d}{dt} \left( \frac{\beta(t)}{\ell(t)} \right) \right) \right).
\end{align*}
\end{enumerate}
\end{fact}

Finally, we recall representations of the $n$-th evolutes of fronts. Let $(\gamma, \nu) : I \rightarrow \mathbb{R}^2 \times \it{S}^1$ be a Legendre immersion without inflection points. We denote $Ev^0(t) = \gamma(t)$ and $Ev^1(t) = Ev(\gamma)(t)$. We define the followings inductively.

\begin{align*}
Ev^n(t) = Ev(Ev^{n-1} (\gamma))(t), \quad \beta_{0}(t) = \beta(t), \quad and \quad \beta_{n}(t) = \dfrac{d}{dt} \Bigl( \dfrac{\beta_{n-1}(t)}{\ell(t)} \Bigl).
\end{align*}
Fukunaga and Takahashi gave the representations of the $n$-th evolutes of fronts.

\begin{fact}(\cite[Theorem 5.1]{Fukunaga and Takahashi})
Let $\gamma : I \rightarrow \mathbb{R}$ be a front. Then the $n$-th evolute of $\gamma$ is given by 

\begin{align}
Ev^n(\gamma)(t) = Ev^{n-1}(\gamma)(t) - \dfrac{\beta_{n-1}(t)}{\ell(t)} M^{n-1}(\nu(t)), \label{align:the n-th evolutes of fronts}
\end{align}
where $M^{n}$ is $n$-times operations of $M$. Moreover, $\left( Ev^{n}(\gamma), M^{n}(\nu) \right)$ is a Legendre immersion with the curvature $\left( \ell(t), \beta_n(t) \right)$.

\end{fact}

\section{\bf Classifications of $(n, n+1)$ cusps in the sense of $C^{1}$-equivalence}
\label{sec:Classification of $(n, n+1)$ cusp in $C^{1}$ class diffeomorphism}

In this section, we prove the following.

\begin{theorem}\label{thm:C1-diffeo-classification}
Let $\gamma : I \rightarrow \mathbb{R}^2$ be a $C^{\infty}$ curve. Then $\gamma$ is $C^{1}$-$equivalent$ to $Cusp_{(n, n+1)}(t) = (t^{n}, t^{n+1})$ at $t = 0$ if and only if
\begin{align}
\gamma'(0) = \gamma''(0) = \cdots = \gamma^{(n-1)}(0) = 0, \label{align:C1-diffeo-classification01}  \\
\det{(\gamma^{(n)}(0), \gamma^{(n+1)}(0))} \neq 0. \label{align:C1-diffeo-classification02}
\end{align}
\end{theorem}
For the proof of this theorem, we divide it into several steps.

\subsection{\bf Invariants of $(n, n+1)$ cusps}

\begin{proposition}\label{prop:coordinate-condition01}
The conditions $(\ref{align:C1-diffeo-classification01})$ and $(\ref{align:C1-diffeo-classification02})$ do not depend on the choices of coordinates on the source nor on the target. 
\end{proposition}

For the proof, it is sufficient to prove the following two lemmas. 

\begin{lemma}\label{lemma:1}

The condition $(\ref{align:C1-diffeo-classification01})$ does not depend on the choices of $C^{\infty}$ local diffeomorphisms $\Phi : (\mathbb{R}^2,0)$ $\rightarrow$ $(\mathbb{R}^2,0)$ nor $\psi : (\mathbb{R},0)$ $\rightarrow$ $(\mathbb{R},0)$. 
\end{lemma}

\begin{lemma}\label{lemma:2}
The condition $(\ref{align:C1-diffeo-classification02})$ does not depend on the choices of $C^{\infty}$ local diffeomorphisms $\Phi : (\mathbb{R}^2,0)$ $\rightarrow$ $(\mathbb{R}^2,0)$ nor $\psi : (\mathbb{R},0)$ $\rightarrow$ $(\mathbb{R},0)$.
\end{lemma}

We will use the following fact to prove Lemma \ref{lemma:1} and Lemma \ref{lemma:2}.

\begin{fact}(Fa\`{a} di Bruno's formula)\label{fact:chain-rule}. Let $f, g : \mathbb{R} \rightarrow \mathbb{R}$ be $C^{\infty}$ functions and $m\in \mathbb{Z}_{>0}$. Then the following holds.

\begin{equation*}
(f \circ g)^{(m)} = {\displaystyle \sum_{k =1}^{m}}  \sum_{\substack{ i_1 + \cdots + i_k = m \\ i_1,\cdots, i_k \neq 0} }\dfrac{m!}{k!} (f^{(k)} \circ g) \prod_{l = 1}^{k} \dfrac{g^{(i_{l})}}{i_{l}!}.
\end{equation*}

\end{fact}

\textit{Proof of Lemma \ref{lemma:1}}.  First we show that the condition $(\ref{align:C1-diffeo-classification01})$ does not depend on the choice of a $C^{\infty}$ local diffeomorphism $\Phi$. We put $\Phi = ( \Phi_{1}, \Phi_{2})$, $\gamma(t) = ( x(t), y(t) )$. Setting $\hat{\gamma}(t) := \Phi \circ \gamma(t) = \Phi( x(t), y(t) )$, we have $\hat{\gamma}'(t)$ = $J(t) \gamma'(t)$, where $J(t) := J(x(t), y(t)) =
\begin{pmatrix}
(\Phi_{1})_{x} (x(t), y(t))&(\Phi_{1})_{y} (x(t), y(t))\\
(\Phi_{2})_{x} (x(t), y(t))&(\Phi_{2})_{y} (x(t), y(t))\\
\end{pmatrix} $. Thus we see that 

\begin{equation}\label{eqn:1}
\hat{\gamma}^{(n)}(t) = \displaystyle \sum_{k = 0}^{n-1} {}_{n-1} C _{k} J(t)^{(n-k-1)} \gamma^{(k+1)}(t).
\end{equation}
By $\gamma'(0) = \gamma''(0) = \cdots = \gamma^{(n-1)}(0) = \bold{0}$, we get 
\begin{align*}
\hat{\gamma}'(0) = \hat{\gamma}''(0) = \cdots = \hat{\gamma}^{(n-1)}(0) = \bold{0}.
\end{align*}
Next we show that the condition $(\ref{align:C1-diffeo-classification01})$ does not depend on the choice of a $C^{\infty}$ local diffeomorphism $\psi$ . We set $\tilde{\gamma}(s) := \gamma \circ \psi(s) = ( x(\psi(s)), y(\psi(s)) )$. Then, by Fact \ref{fact:chain-rule}, we see that 

\begin{align*}
\tilde{\gamma}^{(n)}(0) = \displaystyle \sum_{k=1}^{n} \sum_{\substack{ i_1 + \cdots + i_k = n \\ i_1,\cdots, i_k \neq 0} } \Bigl( \dfrac{n!}{k!} (x^{(k)}(0)) \prod_{l = 1}^{k} \dfrac{\psi^{(i_{l})} (0) }{i_{l}!},  \dfrac{n!}{k!} (y^{(k)} (0)) \prod_{l = 1}^{k} \dfrac{\psi^{(i_{l})} (0)} {i_{l}!} \Bigl).  
\end{align*}
By $\gamma'(0) = \gamma''(0) = \cdots = \gamma^{(n-1)}(0) = \bold{0}$, we get 
\begin{align*}
\tilde{\gamma}'(0) = \hat{\gamma}''(0) = \cdots = \hat{\gamma}^{(n-1)}(0) = \bold{0}.
\end{align*}
Hence we have the assertion.\qed\\

\textit{Proof of Lemma \ref{lemma:2}}. First we show that the condition $(\ref{align:C1-diffeo-classification02})$ does not depend on the choice of a $C^{\infty}$ local diffeomorphism $\Phi$. By (\ref{eqn:1}), we have

\begin{equation}\label{eqn:2}
\hat{\gamma}^{(n+1)}(t) = \displaystyle \sum_{k = 0}^{n} {}_{n} C _{k} J(t)^{(n-k)} \gamma^{(k+1)}(t).
\end{equation}
By $(\ref{eqn:1})$ and $(\ref{eqn:2})$, we get $\hat{\gamma}^{(n)}(0) = J(0) \gamma^{(n)}(0)$ and $\hat{\gamma}^{(n+1)}(t) = {}_{n} C _{n-1} J'(0) \gamma^{(n)}(0) + J(0) \gamma^{(n+1)}(0)$.  Since $J'(t) = J_{x}(t) x'(t) + J_{y}(t) y'(t)$ and $\gamma'(0) = \bold{0}$, we have $J'(0) = \bold{0}$. Thus we get $\hat{\gamma}^{(n+1)}(t) = J(0) \gamma^{(n+1)}(0)$. 
Therefore we obtain
\begin{align*}
\det{(\hat{\gamma}^{(n)}(0), \hat{\gamma}^{(n+1)}(0) )} = \det{J(0)} \det{(\gamma^{(n)}(0), \gamma^{(n+1)}(0))} \neq 0. 
\end{align*}

Next we show that the condition $(\ref{align:C1-diffeo-classification02})$ does not depend on the choice of a local diffeomorphism $\psi$. By Fact \ref{fact:chain-rule}, we have

\begin{align*}
\tilde{\gamma}^{(n+1)}(0) =\displaystyle \sum_{k=1}^{n+1} \sum_{\substack{ i_1 + \cdots + i_k = n+1 \\ i_1,\cdots, i_k \neq 0} } \biggl( \dfrac{(n+1)!}{k!} (x^{(k)}(0)) \prod_{l = 1}^{k} \dfrac{\psi^{(i_{l})} (0) }{i_{l}!},  \dfrac{(n+1)!}{k!} (y^{(k)} (0) ) \prod_{l = 1}^{k} \dfrac{\psi^{(i_{l})} (0)} {i_{l}!}   \biggl). 
\end{align*}
Thus we get 
\begin{align}\label{align:Faa di chain rule parameter}
\tilde{\gamma}^{(n)}(0) = (\psi'(0))^{n} \gamma^{(n)} (0), \quad \tilde{\gamma}^{(n+1)}(0) = C \gamma^{(n)} (0) + (\psi'(0))^{n+1} \gamma^{(n+1)} (0),
\end{align}
where $C$ is some constant. Hence we obtain 
\begin{align*}
\det{(\tilde{\gamma}^{(n)}(0),\tilde{\gamma}^{(n+1)}(0) )} = (\psi'(0))^{2n+1} \det{(\gamma^{(n)}(0), \gamma^{(n+1)}(0))} \neq 0 
\end{align*}
and have the assertion. \qed

\begin{remark}
Proposition \ref{prop:coordinate-condition01} implies that the conditions $(\ref{align:C1-diffeo-classification01})$ and $(\ref{align:C1-diffeo-classification02})$ are the invariants of $(n, n+1)$ cusps in the sense of $\mathcal{A}$-equivalence. Moreover, by using Fact \ref{fact:evolute of one singlarity down} and Corollary \ref{coro:Evolute of regualr points}, we get the following result.
\end{remark}

\begin{theorem}\label{coro:properties of (n,n+1)cusps}
If $\gamma$ is $\mathcal{A}$-$equivalent$ to $Cusp_{n,n+1}(t) = (t^n, t^{n+1})$ at $t=0$, then $t = 0$ is a singular point of $Ev(\gamma), \cdots Ev^{n-2}(\gamma)$ and a regular point of $Ev^{n-1}(\gamma)$.
\end{theorem}

As an application of Proposition \ref{prop:coordinate-condition01}, we can get the following result with respect to $(n,n+1)$-cusps.

\begin{corollary}\label{coro:negative criterion for (n,n+1)-cusps}
Let $\gamma : 0 \in I \rightarrow \mathbb{R}^2$ be a front without inflection points. Assume that $t=0$ is a singular point of $\gamma$, that is, $(d \gamma / dt) (0) = \bold{0}$. If $t=0$ is a singular point of $Ev(\gamma)$, $\cdots$, $Ev^{n-1} (\gamma)$, then $\gamma$ is not $\mathcal{A}$-equivalent to $Cusp_{(n,n+1)}$ at $t=0$.
\end{corollary}

\begin{proof}
Suppose that $\gamma$ is $\mathcal{A}$-equivalent to $Cusp_{(n,n+1)}$ at $t=0$. By Proposition \ref{prop:coordinate-condition01}, $\gamma$ satisfies the conditions $(\ref{align:C1-diffeo-classification01})$ and $(\ref{align:C1-diffeo-classification02})$. 
On the other hand, since $t=0$ is a singular point of $Ev(\gamma)$, $\cdots$, $Ev^{n-1} (\gamma)$, we see that the following by Fact \ref{fact:evolute of one singlarity down}.
\begin{align*}
\frac{d^2 \gamma}{dt^2}(0) = \cdots  = \frac{d^{n} \gamma}{dt^{n}}(0) = \bold{0}.
\end{align*}
This is a contradiction. Therefore we have the assertion.
\end{proof}

\subsection{\bf $C^{1}$ diffeomorphism forms of $(n, n+1)$ cusps}
The purpose of this step is to prove the following proposition.

\begin{proposition}\label{thm:C1-diffeo-cusp-normal-form}
Let $\gamma : I  \rightarrow \mathbb{R}^2$ be a $C^{\infty}$ curve.
Then $\gamma$ satisfies $(\ref{align:C1-diffeo-classification01})$ and $(\ref{align:C1-diffeo-classification02})$ if and only if $\gamma$ is $\mathcal{A}$-equivalent to $C_{(n, n+1)}(s) = (s^n, s^{n+1} Y_{0}(s))$ at $s=0$, where $Y_{0}$ is a $C^{\infty}$ function and $Y_{0}(0) \neq 0$.
\end{proposition}

To prove this proposition, we need the following result (\cite[Proposition A.1]{SUY1}).

\begin{fact}$($The division lemma$)$\label{thm:innsinohodai}
Let $f : U \rightarrow \mathbb{R}$ be a $C^{\infty}$ function defined on a convex open neighborhood $U$ of the origin of $\mathbb{R}^2$. If there exists a non-negative integer $k$ such that 
\begin{equation*}
f(u, 0) = \dfrac{\partial f}{\partial v}(u, 0) = \dfrac{\partial^2 f}{\partial v^2}(u, 0) = \cdots = \dfrac{\partial^k f}{\partial v^k}(u, 0) = 0, 
\end{equation*}
then there exists a  $C^{\infty}$ function $g : U \rightarrow \mathbb{R}$ such that 
\begin{equation*}
f( u, v ) = v^{k + 1} g(u, v), \quad (u,v) \in U. 
\end{equation*}
\end{fact}

\textit{Proof of Proposition \ref{thm:C1-diffeo-cusp-normal-form}}. If $\gamma$ is $\mathcal{A}$-equivalent to $C_{(n, n+1)}(s) = (s^n, s^{n+1} Y_{0}(s))$ at $s=0$, then $C_{(n, n+1)}$ satisfies $(\ref{align:C1-diffeo-classification01})$ and $(\ref{align:C1-diffeo-classification02})$. By Proposition \ref{prop:coordinate-condition01}, $\gamma$ also satisfies $(\ref{align:C1-diffeo-classification01})$ and $(\ref{align:C1-diffeo-classification02})$. 

Assume that $\gamma$ satisfies $(\ref{align:C1-diffeo-classification01})$ and $(\ref{align:C1-diffeo-classification02})$. We put $\gamma(t) = (x(t), y(t))$. By the division lemma, there exist $C^{\infty}$ functions $x_{1}$ and $y_{1}$ such that $x(t) = t^n x_{1}(t)$, $y(t) = t^n y_{1}(t)$. Since $\gamma^{(n)}(0) \neq 0$, we may assume that $x_{1}(0) \neq 0$. In the following, we assume that $x_{1}(t) > 0$ in an open interval around $t = 0$. Define a function $s : \mathbb{R} \rightarrow \mathbb{R}$ by  $s(t) := t \sqrt[n]{x_{1}(t)}$, then $s$  is a $C^{\infty}$ local diffeomorphism on an open interval around $t =0$ (If $x_{1}(t) < 0$, we define $s(t) := t \sqrt[n]{-x_{1}(t)}$). Thus $\gamma$ can be written as $\tilde{\gamma}(s) = (s^n, s^n Y_{1}(s) )$ using $s$, where $Y_{1}(s) := \dfrac{y_{1}(t(s))}{x_{1}(t(s))}$. Define a $C^{\infty}$ function $\tilde{Y}(s) = Y_{1}(s) - Y_{1}(0).$ By using the division lemma, there exists a $C^{\infty}$ function $Y_{0}$ such that $\tilde{Y}(s) = s Y_{0}(s)$ and hence we get $Y_{1}(s) = Y_{1}(0) + s Y_{0}(s)$. Since  $\dfrac{dt}{ds} \neq 0$ at $s=0$, $\det{(\gamma^{(n)}(0), \gamma^{(n+1)}(0))} \neq 0$ and 

\begin{align*}
Y'_{1}(s) = \dfrac{ y'_{1}(t(s)) x_{1}(t(s)) - x'_{1}(t(s)) y_{1}(t(s)) }{ x_{1}(t(s))^2 } \dfrac{dt(s)}{ds},  
\end{align*}
we have $Y'_{1}(0) = Y_{0}(0) \neq 0$. Therefore we obtain 
\begin{align*}
\tilde{\gamma}(s) = ( s^n , s^n( Y_{1}(0) + sY_{0}(s) ) ).
\end{align*}
Define a $C^{\infty}$ diffeomorphism $f : \mathbb{R}^2 \rightarrow \mathbb{R}^2$ by 
\begin{align*}
f(u,v) = ( u , v - uY_{1}(0) ).
\end{align*}
Then $f \circ \tilde{\gamma}(s) = (s^n, s^{n+1} Y_{0}(s) )$. This concludes the proof. \qed

\subsection{\bf Proof of Theorem \ref{thm:C1-diffeo-classification}}
If $\gamma = Cusp_{(n,n+1)}$, then it is clear that the conditions $(\ref{align:C1-diffeo-classification01})$ and $(\ref{align:C1-diffeo-classification02})$ hold. 
Hence, by Proposition \ref{prop:coordinate-condition01}, if $\gamma$ is $C^{1}$-equivalent to $Cusp_{(n,n+1)}$ at $t=0$, then the conditions $(\ref{align:C1-diffeo-classification01})$ and $(\ref{align:C1-diffeo-classification02})$ hold.

Conversely, assume that $\gamma$ satisfies$(\ref{align:C1-diffeo-classification01})$ and $(\ref{align:C1-diffeo-classification02})$. By Proposition \ref{thm:C1-diffeo-cusp-normal-form}, $\gamma$ is $\mathcal{A}$-equivalent to $C_{(n, n+1)}(s) = (s^n, s^{n+1} Y_{0}(s))$ at $s=0$, where $Y_{0}$ is a $C^{\infty}$ function and $Y_{0}(0) \neq 0$. Define a map $\phi : \mathbb{R}^2 \rightarrow \mathbb{R}^2$ by $\phi(x, y) := (x, z) $, where $z$ is the uniquely-determined $C^1$ function of $y$ satisfying $\displaystyle y = z Y_0 (z^{\frac{1}{n+1}})$. Then we see that

\begin{equation*}
\dfrac{dy}{dz} = \displaystyle Y_{0}(z^{\frac{1}{n+1}}) + z \frac{dY_{0}}{ds}(z^{\frac{1}{n+1}}) \frac{1}{n+1} z^{\frac{1}{n+1}-1} = \displaystyle Y_{0}(z^{\frac{1}{n+1}}) + \frac{1}{n+1} z^{\frac{1}{n+1}} \frac{dY_{0}}{ds}(z^{\frac{1}{n+1}}).
\end{equation*}
By $\dfrac{dy(0)}{dz}= Y_{0}(0) \neq 0$, $y$ is a $C^{1}$ local diffeomorphism on an open interval around $z=0$. Thus $\phi$ is a $C^{1}$ local diffeomorphism on an open neighborhood of the origin of $\mathbb{R}^2$. In particular, note that the only $z$ corresponding to $y = s^{n+1} Y_{0}(s) $ is $z = s^{n+1}$. Therefore we obtain $\phi(s^n, s^{n+1} Y_{0}(s)) = (s^n, s^{n+1})$ and have the assertion. \qed　

\section{\bf Criterion for $(4, 5)$-cusp}\label{sec:criterion for (4, 5)cusp}

In this section, we give a necessary and sufficient condition for a plane curve to be $\mathcal{A}$-$equivalent$ to $Cusp_{(4,5)}$ at the origin. Namely, we prove the following. 

\begin{theorem}\label{thm:criterion for (4, 5)cusp}
Let $\gamma : I \rightarrow \mathbb{R}^2$ be a $C^{\infty}$ curve. Then $\gamma$ is $\mathcal{A}$-equivalent to $Cusp_{(4, 5)}(t) =(t^4, t^5)$ at $t=0$ if and only if $\gamma'(0) = \gamma''(0) = \gamma'''(0) = \bold{0}$, $A \neq 0$ and 

\begin{align}
-77B^2 + 105AD + 60AC = 0,
\end{align}\label{eqn:(4,5)cusp criterion condition}
where $A = \det{(\gamma^{(5)}(0) , \gamma^{(4)}(0))}$, $B = \det{(\gamma^{(6)}(0), \gamma^{(4)}(0))}$, $C = \det{(\gamma^{(7)}(0), \gamma^{(4)}(0))}$ and $D = \det{(\gamma^{(6)}(0), \gamma^{(5)}(0))}$. 
\end{theorem}

Before we prove this theorem, we check the following.

\begin{proposition}\label{prop:coordinate-condition10}
Under the conditions $\gamma'(0) = \gamma''(0) = \gamma'''(0) = \bold{0}$, the condition $(\ref{eqn:(4,5)cusp criterion condition})$ does not depend on the choices of coordinates on the source nor on the target. 
\end{proposition}

\begin{proof}
First we show that the condition $(\ref{eqn:(4,5)cusp criterion condition})$ does not depend on the choice of a $C^{\infty}$ local diffeomorphism $\Phi : (\mathbb{R}^2,0)$ $\rightarrow$ $(\mathbb{R}^2,0)$. We put $\Phi = (\Phi_{1}, \Phi_{2})$ and $\gamma(t) = (x(t), y(t))$. Moreover, we set $\hat{\gamma}(t) = \Phi \circ \gamma(t) = \Phi( x(t), y(t) )$. By (\ref{eqn:1}), we have the followings. 
\begin{align*}
\hat{\gamma}^{(4)}(t) = \displaystyle \sum_{k = 0}^{3} {}_{3} C _{k} J(t)^{(3-k)} \gamma^{(k+1)}(t), \quad \hat{\gamma}^{(5)}(t) = \displaystyle \sum_{k = 0}^{4} {}_{4} C _{k} J(t)^{(4-k)} \gamma^{(k+1)}(t),   \\
\hat{\gamma}^{(6)}(t) = \displaystyle \sum_{k = 0}^{5} {}_{5} C _{k} J(t)^{(5-k)} \gamma^{(k+1)}(t), \quad \hat{\gamma}^{(7)}(t) = \displaystyle \sum_{k = 0}^{6} {}_{6} C _{k} J(t)^{(6-k)} \gamma^{(k+1)}(t). 
\end{align*}
Note that we see that $J'(0) = J''(0) = J'''(0) = \bold{0}$ by $\gamma'(0) = \gamma''(0) = \gamma'''(0) = \bold{0}$. Thus we have the followings.

\begin{align*}
\hat{\gamma}^{(4)}(0) = J(0) \gamma^{(4)}(0), \quad \hat{\gamma}^{(5)}(0) = J(0)\gamma^{(5)}(0),  \\
\hat{\gamma}^{(6)}(0) = J(0)\gamma^{(6)}(0), \quad \hat{\gamma}^{(7)}(0) = J(0)\gamma^{(7)}(0). 
\end{align*}
Setting $\hat{A} = \det{(\hat{\gamma}^{(5)}(0) , \hat{\gamma}^{(4)}(0))}$, $\hat{B} = \det{(\hat{\gamma}^{(6)}(0), \hat{\gamma}^{(4)}(0))}$, $\hat{C} = \det{(\hat{\gamma}^{(7)}(0), \hat{\gamma}^{(4)}(0))}$ and $\hat{D} = \det{(\hat{\gamma}^{(6)}(0), \hat{\gamma}^{(5)}(0))}$, we have 

\begin{align}\label{align:invariant of (4,5)cusp part1 }
-77 \hat{B}^2 + 105 \hat{A} \hat{D} + 60 \hat{A} \hat{C} = (\det{J(0)})^2 (-77B^2 + 105AD + 60AC) = 0.
\end{align}

Next we show that the condition $(\ref{eqn:(4,5)cusp criterion condition})$ does not depend on the choice of a $C^{\infty}$ local diffeomorphism $\psi : (\mathbb{R},0)$ $\rightarrow$ $(\mathbb{R},0)$. We set $\tilde{\gamma}(t) = \gamma(\psi(t))$. By $\gamma'(0) = \gamma''(0) = \gamma'''(0) = \bold{0}$, $\gamma(t)$ can be written as

\begin{align}
\gamma(t) = (t^4 z(t) , t^4 w(t)),
\end{align}
where $z(t)$ and $w(t)$ are $C^{\infty}$ functions. Then we have the followings.

\begin{align*}
A &= -2880 z(0) w'(0) + 2880 w(0) z'(0),  \\
B &= -8640 z(0) w''(0) + 8640 w(0)z''(0),  \\
C &= -20160 z(0) w'''(0) + 20160 w(0) z'''(0),  \\
D &= -43200 z'(0) w''(0) + 43200 w'(0) z''(0).  
\end{align*}
Thus we see that 

\begin{align*}
-77B^2 + 105AD + 60AC &= -5748019200 ( -z(0) w''(0) + w(0)z''(0))^2 \\
&- 13063680000 (z(0) w'(0) - w(0) z'(0)) ( -z'(0) w''(0) + w'(0) z''(0) )  \\
&+ 3483648000 ( z(0) w'(0) - w(0) z'(0) ) ( z(0) w'''(0) - w(0) z'''(0) ). 
\end{align*}
Setting $\tilde{A} = \det{(\tilde{\gamma}^{(5)}(0) , \tilde{\gamma}^{(4)}(0))}$, $\tilde{B} = \det{(\tilde{\gamma}^{(6)}(0), \tilde{\gamma}^{(4)}(0))}$, $\tilde{C} = \det{(\tilde{\gamma}^{(7)}(0), \tilde{\gamma}^{(4)}(0))}$ and $\tilde{D} = \det{(\tilde{\gamma}^{(6)}(0), \tilde{\gamma}^{(5)}(0))}$, we obtain the following by the same calculations.

\begin{align}\label{eqn:parameter change part001}
-77\tilde{B}^2 + 105\tilde{A} \tilde{D} + 60 \tilde{A} \tilde{C} = (\psi'(0))^{20} ( -77B^2 + 105AD + 60AC ) = 0. 
\end{align}
Therefore we have the assertion. 
\end{proof}

\begin{remark}\label{remark:invariant of sign}
By (\ref{align:invariant of (4,5)cusp part1 }) and (\ref{eqn:parameter change part001}), we see that the expression $-77B^2 + 105AD + 60AC$ does not change sign by coordinate transformations on the source nor on the target. 
\end{remark}

\begin{definition}\label{definition:new (4,5) cusipdal}
Under the conditions $\gamma'(0) = \gamma''(0) = \gamma'''(0) = \bold{0}$ and $\gamma^{(4)}(0) \neq \bold{0}$, we define the \it{$(4, 5; \pm 7)$-cuspidal curvature} as follow:
\begin{equation}
\kappa_{q} = \dfrac{-77B^2 + 105AD + 60AC}{\| \gamma^{(4)}(0) \|^5 }, 
\end{equation}
where $A = \det{(\gamma^{(5)}(0) , \gamma^{(4)}(0))}$, $B = \det{(\gamma^{(6)}(0), \gamma^{(4)}(0))}$, $C = \det{(\gamma^{(7)}(0), \gamma^{(4)}(0))}$ and $D = \det{(\gamma^{(6)}(0), \gamma^{(5)}(0))}$.
\end{definition}

\begin{proposition}
The number $\kappa_{q}$ does not depend on the choices of parameter transformations on the source nor of congruent transformations on the target. 
\end{proposition}

\begin{proof}
By (\ref{align:Faa di chain rule parameter}) and (\ref{eqn:parameter change part001}), we see that the number $\kappa_{q}$ is an invariant with respect to parameter transformations. It is clear that the number $\kappa_{q}$ is independent of congruent transformations by (\ref{align:invariant of (4,5)cusp part1 }). 
\end{proof}

The $(4, 5; \pm 7)$-cuspidal curvature may depend on a coordinate transformation on the target. However, this never changes the sign by coordinate transformations.

\subsection{\bf $(4, 5)$-cusps and Whitney's lemma}

The purpose of this step is to prove the following proposition.

\begin{proposition}\label{prop:(4,5)cusp form}
$\alpha(s) = (s^{4}, s^{5}) + M(s)$ is $\mathcal{A}$-equivalent to $Cusp_{(4, 5)}(s) = (s^4, s^5)$ at $s=0$, where $M(s)$ is a $\mathbb{R}^2$-valued $C^{\infty}$ function satisfying  $M(0) = M'(0) \cdots = M^{(7)}(0) = \bold{0}$ and $M^{(11)}(0) = \bold{0}$. 
\end{proposition}

To prove this proposition, we need the following result (\cite[Theorem 3.1.7]{SUY1}, \cite{Whitney} ).  

\begin{fact}(Whitney's lemma)\label{thm:Whitney lemma}
Let $f : U \subset \mathbb{R}^2 \rightarrow \mathbb{R}$ be a $C^{\infty}$ function defined on an open neighborhood of the origin of $\mathbb{R}^2$. If $f$ satisfies $f(u, v) = f(u, -v)$, then there exists a $C^{\infty}$ function $g$ defined on an open neighborhood of the origin of $\mathbb{R}^2$ such that $f(u, v) = g(u, v^2)$ holds near the origin. 
\end{fact}

\begin{lemma}\label{lem:General Whintney lemma}
Let $k$ be an integer greater than or equal to $1$. Then for any $C^{\infty}$ function $f : \mathbb{R} \rightarrow \mathbb{R}$, there exist $C^{\infty}$ functions $g_{l} : I \subset \mathbb{R} \rightarrow \mathbb{R}$ $(l = 1,2,\cdots, 2^{k})$ such that 

\begin{equation}\label{eq:General Whintney formula}
\displaystyle f(t) = \sum_{l=1}^{2^k} { t^{l-1} g_{l}(t^{2^k}) }, 
\end{equation}
where $I$ is an open interval around $t=0$.
\end{lemma}

\begin{proof}
We use the mathematical induction and Whitney's lemma. In fact, we difine functions 
\begin{equation*}
u(t) := \dfrac{f(t) + f(-t)}{2},\quad  v(t) := \dfrac{f(t) - f(-t)}{2}. 
\end{equation*}
Since $u(t)$ is an even function, there exists a $C^{\infty}$ function $g_{1}$ such that $u(t) = g_{1}(t^2)$ by Whitney's lemma. Moreover, since $v(t)$ is an odd function, there exists a $C^{\infty}$ function $g_{2}$ such that $u(t) = tg_{2}(t^2)$ by the division lemma and  Whitney's lemma. Note that $f(t) = u(t) + v(t)$. Therefore we see that

\begin{equation*}
f(t) = g_{1}(t^2) + t g_{2}(t^2) 
\end{equation*}
holds on an open interval containing $t=0$. By the same argument as above for $g_{1}$ and $g_{2}$, we have the assertion.
\end{proof}

\textit{Proof of Proposition \ref{prop:(4,5)cusp form}}. By Lemma \ref{lem:General Whintney lemma}, there exist $\mathbb{R}^2$-valued $C^{\infty}$ functions $M_{1}(s)$, $M_{2}(s)$, $M_{3}(s)$ and $M_{4}(s)$ defined on an open interval containing $s=0$ such that  

\begin{align}\label{align:Whitney form 4,5cusp}
M(s) = M_{1}(s^4) + sM_{2}(s^4) + s^2 M_{3}(s^4) + s^3 M_{4}(s^4).
\end{align}
Since $M(s)$ satisfies $M(0) = M'(0) = \cdots = M^{(7)}(0) = \bold{0}$ and $M^{(11)}(0) = \bold{0}$, there exist $\mathbb{R}^2$-valued $C^{\infty}$ functions $\tilde{M}_{1}(s)$, $\tilde{M}_{2}(s)$, $\tilde{M}_{3}(s)$ and $\tilde{M}_{4}(s)$ defined on an open interval containing  $s=0$ such that 

\begin{align*}
M_{i}(s) = s^2 \tilde{M}_{i}(s), M_{4}(s) = s^3 \tilde{M}_{4}(s), \quad (i = 1, 2, 3).
\end{align*}
Substituting these results into (\ref{align:Whitney form 4,5cusp}), we get the following.

\begin{align*}\label{align: Whintney form 4,5cusp No2}
M(s) = s^8 \tilde{M}_{1}(s^4) + s^9 \tilde{M}_{2}(s^4) + s^{10} \tilde{M}_{3}(s^4) + s^{15} \tilde{M}_{4}(s^4). 
\end{align*}
Define a $C^{\infty}$ map $\psi : \mathbb{R}^2 \rightarrow \mathbb{R}^2$ by 
\begin{align*}
\psi(X, Y) = X^2 \tilde{M}_{1}(X) +  XY \tilde{M}_{2}(X) +  Y^2 \tilde{M}_{3}(X) +  Y^3 \tilde{M}_{4}(X).
\end{align*}
Then we see that $\psi(s^4, s^5) = M(s)$. Define a $C^{\infty}$ diffeomorphism $\Psi : \mathbb{R}^2 \rightarrow \mathbb{R}^2$ by 
\begin{align*}
\Psi(X, Y) = (X, Y) + \psi(X, Y). 
\end{align*}
Therefore we get  
\begin{align}
\Psi(s^4, s^5) = (s^4, s^5) + M(s) = \alpha(s). 
\end{align}
This concludes the proof. 
\qed

\subsection{\bf Eliminations of the $7$-th order and the $11$-th order terms}

The essential problems of a criterion for $(4, 5)$-cusps are the treatments of the $7$-th order and $11$-th order terms. It is due to the fact that $7$, $11$ $\notin \{ 4p + 5q$ ; $p, q \in \mathbb{Z}_{>0} \}$. In this step, we discuss the method of them. 

\begin{lemma}\label{lem:normal form of smooth (n,n+1)cusp}
Let $\gamma : I \rightarrow \mathbb{R}^2$ be a $C^{\infty}$ curve. Then $\gamma$ satisfies $(\ref{align:C1-diffeo-classification01})$ and $(\ref{align:C1-diffeo-classification02})$ if and only if $\gamma$ is $\mathcal{A}$-equivalent to $\gamma_{0}(t) = (t^n , t^{n+1}) + h_{n+3}(t)$ at $t=0$, where $h_{n+3}$ is a $\mathbb{R}^2$-valued $C^{\infty}$ function satisfying $h_{n+3}(0) = h_{n+3}'(0) \cdots = h_{n+3}^{(n+2)}(0) = \bold{0}$. 
\end{lemma}

\begin{proof}
If $\gamma$ is $\mathcal{A}$-equivalent to $\gamma_{0}(t) = (t^n , t^{n+1}) + h_{n+3}(t)$ at $t=0$, then $\gamma$ satisfies $(\ref{align:C1-diffeo-classification01})$ and $(\ref{align:C1-diffeo-classification02})$,  by Proposition \ref{prop:coordinate-condition01}. 

Conversely, assume that $\gamma$ satisfies $(\ref{align:C1-diffeo-classification01})$ and $(\ref{align:C1-diffeo-classification02})$. By $\gamma'(0) = \gamma''(0) = \cdots = \gamma^{(n-1)}(0) = \bold{0}$, $\gamma$ can be written as 

\begin{align*}
\gamma(t) = (a_{n} t^n + a_{n+1} t^{n+1} + a_{n+2} t^{n+2}, b_{n} t^n + b_{n+1} t^{n+1} + b_{n+2} t^{n+2}) + H_{n + 3} (t), 
\end{align*}
where $H_{n+3}$ is a $\mathbb{R}^2$-valued $C^{\infty}$ function satisfying $H_{n+3}(0) = H_{n+3}'(0) \cdots = H_{n+3}^{(n+2)}(0) = \bold{0}$. Since $\det{(\gamma^{(n)}(0), \gamma^{(n+1)}(0))} \neq 0$, there exists a linear coordinate transformation of $\mathbb{R}^2$ such that $\gamma$ is $\mathcal{A}$-equivalent to

\begin{align}\label{align:General (n,n+1)cusp formation}
\gamma_{L}(t) = (t^n + \hat{a}_{n+2} t^{n+2} , t^{n+1} + \hat{b}_{n+2} t^{n+2} ) + \hat{H}_{n+3}(t),
\end{align}
where $\hat{H}_{n+3}$ is a $\mathbb{R}^2$-valued $C^{\infty}$ function satisfying $\hat{H}_{n+3}(0) = \hat{H}_{n+3}'(0) \cdots = \hat{H}_{n+3}^{(n+2)}(0) = \bold{0}$. Define a parameter transformation $\tau : \mathbb{R} \rightarrow \mathbb{R}$ by

\begin{align}\label{align:General transformation}
\tau(\tilde{t}) = \tilde{t} + c_{1}\tilde{t}^2 + c_{2} \tilde{t}^3,
\end{align}
where $c_{1} := - \dfrac{\hat{b}_{n+2}}{n+1} $, $c_{2} := - \dfrac{1}{n} \Bigl( \hat{a}_{n+2} +  \dfrac{n(n-1)}{2(n+1)^2} \hat{b}^2 _{n+2} \Bigl)$. By (\ref{align:General (n,n+1)cusp formation}) and (\ref{align:General transformation}), we obtain 

\begin{equation*}\label{equation:(n, n+1)cusp parameter change}
\gamma_{L} \circ \tau(\tilde{t}) = \Bigl( \tilde{t}^n - \dfrac{n}{n+1}\hat{b}_{n+2} \tilde{t}^{n+1}, \tilde{t}^{n+1} \Bigl) + \tilde{H}_{n+3}(\tilde{t}),
\end{equation*}
where $\tilde{H}_{n+3}$ is a $\mathbb{R}^2$-valued $C^{\infty}$ function satisfying $\tilde{H}_{n+3}(0) = \tilde{H}_{n+3}'(0) \cdots =\tilde{H}_{n+3}^{(n+2)}(0) = \bold{0}$. Define a  $C^{\infty}$ diffeomorphism $\tilde{f} : \mathbb{R}^2 \rightarrow \mathbb{R}^2$ by 

\begin{align*}
\tilde{f}(x,y) = \Bigl( x + \dfrac{n}{n+1} \hat{b}_{n+2} y , y \Bigl). 
\end{align*}
Composing $\gamma_{L} \circ \tau$ and $\tilde{f}$, we have

\begin{align}\label{cuspidal:2}
\tilde{f} \circ \gamma_{L} \circ \tau(\tilde{t}) = (\tilde{t}^n , \tilde{t}^{n+1}) + h_{n+3}(\tilde{t}), 
\end{align}
where $h_{n+3}$ is a $\mathbb{R}^2$-valued $C^{\infty}$ function satisfying $h_{n+3}(0) = h_{n+3}'(0) \cdots = h_{n+3}^{(n+2)}(0) = \bold{0}$. Therefore we have the assertion.
\end{proof}

\begin{corollary}\label{lem:normal form of smooth (n,n+1)cusp part2}
Let $\gamma : I \rightarrow \mathbb{R}^2$ be a $C^{\infty}$ curve. Then $\gamma$ satisfies  $(\ref{align:C1-diffeo-classification01})$ and $(\ref{align:C1-diffeo-classification02})$ if and only if there exists a real number $T$ such that $\gamma$ is $\mathcal{A}$-equivalent to $\tilde{\gamma}(s) = (s^n , s^{n+1} + Ts^{n + 3}) + h_{n+4}(s)$ at $s=0$, where $h_{n+4}$ is a $\mathbb{R}^2$-valued $C^{\infty}$ function satisfying $h_{n+4}(0) = h_{n+4}'(0) \cdots = h_{n+4}^{(n+3)}(0) = \bold{0}$. 
\end{corollary}

\begin{proof}
If $\gamma$ is $\mathcal{A}$-equivalent to $\tilde{\gamma}(s) = (s^n , s^{n+1} + Ts^{n + 3}) + h_{n+4}(s)$ at $s=0$, then $\gamma$ satisfies $(\ref{align:C1-diffeo-classification01})$ and $(\ref{align:C1-diffeo-classification02})$, by Proposition \ref{prop:coordinate-condition01}.

We assume that $\gamma$ satisfies $(\ref{align:C1-diffeo-classification01})$ and $(\ref{align:C1-diffeo-classification02})$. By Lemma \ref{lem:normal form of smooth (n,n+1)cusp}, $\gamma$ is $\mathcal{A}$-equivalent to $\gamma_{0}(t) = (t^n , t^{n+1}) + h_{n+3}(t)$ at $t=0$, where $h_{n+3}$ is a $\mathbb{R}^2$-valued $C^{\infty}$ function satisfying $h_{n+3}(0) = h_{n+3}'(0) \cdots = h_{n+3}^{(n+2)}(0) = \bold{0}$. Thus we can write 

\begin{align*}
h_{n+3}(t) = ( \tilde{a}_{n+3} t^{n+3} + \cdots, \tilde{b}_{n+3} t^{n+3} + \cdots).
\end{align*}
Define a parameter transformation $\phi : \mathbb{R} \rightarrow \mathbb{R}$ by 

\begin{align*}
\phi(s) = s - \dfrac{\tilde{a}_{n+3}}{n}s^4. 
\end{align*}
Composing $\gamma_{0}$ and $\phi$, we obtain

\begin{align}
\gamma_{0} \circ \phi(s) = (s^n, s^{n+1} + \tilde{b}_{n+3}s^{n+3}) + \tilde{h}_{n+4}(s), 
\end{align}
where $\tilde{h}_{n+4}$ is a $\mathbb{R}^2$-valued $C^{\infty}$ function satisfying $\tilde{h}_{n+4}(0) = \tilde{h}_{n+4}'(0) \cdots = \tilde{h}_{n+4}^{(n+3)}(0) = \bold{0}$. Setting $T = \tilde{b}_{n+3}$ and $h_{n+4}(s) = \tilde{h}_{n+4}(s)$, we have the assertion.
\end{proof}

Now, we are ready to discuss the elimination of the $7$-th order and the $11$-th order terms. The following result plays an essential role in the proof of a criterion for $(4, 5)$-cusps given by Lemma \ref{lemma:coeffient of 11}.  

\begin{lemma}\label{lem:coeffient of 7}
Let $\gamma : I \rightarrow \mathbb{R}^2$ be a $C^{\infty}$ curve. Then $\gamma$ satisfies $\gamma'(0) = \gamma''(0) = \gamma'''(0) = \bold{0}$, $A \neq 0$ and $(\ref{eqn:(4,5)cusp criterion condition})$ if and only if $\gamma$ is $\mathcal{A}$-equivalent to a $\gamma_{1}(s) = (s^4, s^5) + h_{8}(s)$ at $s=0$, where $h_{8}$ is a $\mathbb{R}^2$-valued $C^{\infty}$ function satisfying $h_{8}(0) = h_{8}'(0) = \cdots = h_{8}^{(7)}(0) = \bold{0}$.
\end{lemma}

\begin{proof}
If $\gamma$ is $\mathcal{A}$-equivalent to $\gamma_{1}(s) = (s^4, s^5) + h_{8}(s)$ at $s=0$, then $\gamma$ satisfies $\gamma'(0) = \gamma''(0) = \gamma'''(0) = \bold{0}$, $A \neq 0$ and (\ref{eqn:(4,5)cusp criterion condition}) by Propositions \ref{prop:coordinate-condition01} and \ref{prop:coordinate-condition10}. 

Conversely, assume that $\gamma$ satisfies $\gamma'(0) = \gamma''(0) = \gamma'''(0) = \bold{0}$, $A \neq 0$ and (\ref{eqn:(4,5)cusp criterion condition}). By Corollary \ref{lem:normal form of smooth (n,n+1)cusp part2}, there exists a real number $T$ such that $\gamma$ is $\mathcal{A}$-equivalent to $\tilde{\gamma}(s) = (s^4, s^5 + T s^7) + h_{8}(s)$ at $s=0$. Calculating the expression $-77B^2 + 105AD + 60AC$ for $\gamma_{1}$, we have 

\begin{align}\label{align; condition wtr T}
-77B^2 + 105AD + 60AC = 20901888000 T. 
\end{align}
Thus we get T = 0 and see that $\gamma$ is $\mathcal{A}$-equivalent to a $\gamma_{1}$ at $s=0$. This completes the proof.
\end{proof}

\begin{lemma}\label{lemma:coeffient of 11}
Let $\gamma : I \rightarrow \mathbb{R}^2$ be a $C^{\infty}$ curve. Then $\gamma$ satisfies $\gamma'(0) = \gamma''(0) = \gamma'''(0) = \bold{0}$, $A \neq 0$ and $(\ref{eqn:(4,5)cusp criterion condition})$ if and only if $\gamma$ is $\mathcal{A}$-equivalent to $\gamma_{2}(s) = (s^4, s^5) + K_{8}(s)$ at $s=0$, where $K_{8}$ is a $\mathbb{R}^2$-valued $C^{\infty}$ function satisfying $K_{8}(0) = K_{8}'(0) = \cdots = K_{8}^{(7)}(0) = \bold{0}$ and $K_{8}^{(11)}(0) = \bold{0}$.
\end{lemma}

\begin{proof}
If $\gamma$ is $\mathcal{A}$-equivalent to $\gamma_{2}(s) = (s^4, s^5) + K_{8}(s)$ at $s=0$, then clearly $\gamma$ satisfies  $\gamma'(0) = \gamma''(0) = \gamma'''(0) = \bold{0}$, $A \neq 0$ and $(\ref{eqn:(4,5)cusp criterion condition})$ by Lemma \ref{lem:coeffient of 7}. 

Assume that $\gamma$ satisfies $\gamma'(0) = \gamma''(0) = \gamma'''(0) = \bold{0}$, $A \neq 0$ and $(\ref{eqn:(4,5)cusp criterion condition})$. By Lemma \ref{lem:coeffient of 7}, $\gamma$ is $\mathcal{A}$-equivalent to a $\gamma_{1}(s) = (s^4, s^5) + h_{8}(s)$, where $h_{8}(t) = ( \tilde{a}_{8} t^{8} + \cdots, \tilde{b}_{8} t^{8} + \cdots)$. Define a coordinate transformation $\phi : \mathbb{R} \rightarrow \mathbb{R}$ by 

\begin{align}
\phi(s) =  s  - \dfrac{\tilde{b}_{11}}{5} s^7 - \dfrac{\tilde{a}_{11}}{4} s^8,
\end{align}
Composing $\gamma_{1}$ and $\phi$, the coefficients of the $11$-th order term are $(0, 0)$. Therefore we have the assertion.
\end{proof}

\subsection{\bf Proof of Theorem \ref{thm:criterion for (4, 5)cusp}}

If $\gamma$ satisfies $\gamma'(0) = \gamma''(0) = \gamma'''(0) = \bold{0}$, $A \neq 0$ and $(\ref{eqn:(4,5)cusp criterion condition})$, then $\gamma$ is $\mathcal{A}$-equivalent to $\gamma_{2}(s) = (s^4, s^5) + K_{8}(s)$ at $s=0$ by Lemma \ref{lemma:coeffient of 11}. By Proposition \ref{prop:(4,5)cusp form}, we see that $\gamma$ is also $\mathcal{A}$-equivalent to $Cusp_{(4,5)}$ at $s=0$. 

Assume that $\gamma$ is $\mathcal{A}$-equivalent to $Cusp_{(4,5)}$ at $s=0$. Using Proposition \ref{prop:coordinate-condition01} and Proposition \ref{prop:coordinate-condition10}, we get the conclusion. \qed

\begin{theorem}\label{thm:criterion of other type}
Let $\gamma : I \rightarrow \mathbb{R}^2$ be a $C^{\infty}$ curve.

\begin{itemize}
\item[(1)] $\gamma$ is $\mathcal{A}$-equivalent to $Cusp_{(4, 5; +7)}(t) = (t^4, t^5 + t^7)$ at $t=0$ if and only if $\gamma'(0) = \gamma''(0) = \gamma'''(0) = \bold{0}$, $A \neq 0$ and $-77B^2 + 105AD + 60AC > 0$,
\item[(2)]$\gamma$ is $\mathcal{A}$-equivalent to $Cusp_{(4, 5; -7)}(t) =(t^4, t^5 - t^7)$ at $t=0$ if and only if $\gamma'(0) = \gamma''(0) = \gamma'''(0) = \bold{0}$, $A \neq 0$ and $-77B^2 + 105AD + 60AC < 0$,
\item[(3)] $Cusp_{(4, 5; +7)}$ and $Cusp_{(4, 5; -7)}$ are not $\mathcal{A}$-equivalent at $t=0$,
\end{itemize}
where $A = \det{(\gamma^{(5)}(0) , \gamma^{(4)}(0))}$, $B = \det{(\gamma^{(6)}(0), \gamma^{(4)}(0))}$, $C = \det{(\gamma^{(7)}(0), \gamma^{(4)}(0))}$ and $D = \det{(\gamma^{(6)}(0), \gamma^{(5)}(0))}$. 
\end{theorem}

\begin{proof}
By Corollary \ref{lem:normal form of smooth (n,n+1)cusp part2}, $\gamma$ satisfies $\gamma'(0) = \gamma''(0) = \gamma'''(0) = \bold{0}$ and $A \neq 0$ if and only if there exists a real number $T$ such that $\gamma$ is $\mathcal{A}$-equivalent to $\tilde{\gamma}(s) = (s^4 , s^{5} + Ts^{7}) + h_{8}(s)$ at $s=0$, where $h_{8}$ is a $\mathbb{R}^2$-valued $C^{\infty}$ function satisfying $h_{8}(0) = h_{8}'(0) \cdots = h_{8}^{(7)}(0) = \bold{0}$. We note that $T > 0$ (resp. $T < 0$) if and only if $-77B^2 + 105AD + 60AC > 0$ (resp. $-77B^2 + 105AD + 60AC < 0$) by (\ref{align; condition wtr T}).
\begin{itemize}
\item[(1)] By the above discussion, $\gamma$ satisfies $\gamma'(0) = \gamma''(0) = \gamma'''(0) = \bold{0}$, $A \neq 0$ and $-77B^2 + 105AD + 60AC > 0$ if and only if there exists a positive real number $T$ such that $\gamma$ is $\mathcal{A}$-equivalent to a $\tilde{\gamma}(s) = (s^4 , s^{5} + Ts^{7}) + h_{8}(s)$ at $s=0$. We set $p = \sqrt{T}$ and 
\begin{equation*}
L = \begin{pmatrix}{}
p^4 &0\\
 0 & p^5 \\
\end{pmatrix}.
\end{equation*}
Then we see that 

\begin{equation*}
L \circ \tilde{\gamma} \Bigl( \dfrac{s}{p} \Bigr)  =  (s^4, s^5 + s^7) = Cusp_{(4, 5; +7)}(s). 
\end{equation*}

If $\gamma$ is $\mathcal{A}$-equivalent to $Cusp_{(4, 5, +7)}$ at $t=0$, then $\gamma$ satisfies $\gamma'(0) = \gamma''(0) = \gamma'''(0) = \bold{0}$, $A \neq 0$ and $-77B^2 + 105AD + 60AC > 0$ by Proposition \ref{prop:coordinate-condition01} and Remark \ref{remark:invariant of sign}. 

\item[(2)]By a similar argument to (1), we get the assertion.

\item[(3)] Using Remark \ref{remark:invariant of sign}, we immediately have the conclusion. 
\end{itemize}
\end{proof}

Summarizing Theorem \ref{thm:C1-diffeo-classification},  Theorem \ref{thm:criterion for (4, 5)cusp} and Theorem \ref{thm:criterion of other type} by using $\kappa_{q}$, we obtain the following. 

\begin{theorem}\label{thm: complete classifications of (4, 5)cusps }
Let $\gamma : I \rightarrow \mathbb{R}^2$ be a $C^{\infty}$ curve.

\begin{itemize}
\item[(1)] $\gamma$ is $\mathcal{A}$-equivalent to $Cusp_{(4, 5)}(t) = (t^4, t^5)$ at $t=0$ if and only if $\gamma'(0) = \gamma''(0) = \gamma'''(0) = \bold{0}$, $A \neq 0$ and $\kappa_{q} = 0$.
\item[(2)] $\gamma$ is $\mathcal{A}$-equivalent to $Cusp_{(4, 5; +7)}(t) = (t^4, t^5 + t^7)$ at $t=0$ if and only if $\gamma'(0) = \gamma''(0) = \gamma'''(0) = \bold{0}$, $A \neq 0$ and $\kappa_{q} > 0$.
\item[(3)]$\gamma$ is $\mathcal{A}$-equivalent to $Cusp_{(4, 5; -7)}(t) =(t^4, t^5 - t^7)$ at $t=0$ if and only if $\gamma'(0) = \gamma''(0) = \gamma'''(0) = \bold{0}$, $A \neq 0$ and $\kappa_{q} < 0$.
\item[(4)] $Cusp_{(4, 5)}$, $Cusp_{(4, 5; +7)}$ and $Cusp_{(4, 5; -7)}$ are not $\mathcal{A}$-equivalent at $t=0$. However, they are $C^{1}$-equivalent at $t=0$.
\end{itemize}
\end{theorem}

\begin{remark}
$Cusp_{(4, 5; +7)}$ and $Cusp_{(4, 5; -7)}$ are complex-analytically equivalent at $t = 0$. In fact, we set a matrix

\begin{equation*}
\mathcal{L} = \begin{pmatrix}{}
1 &0\\
 0 & - i \\
\end{pmatrix}, \quad  ( i = \sqrt{-1} ). 
\end{equation*}
Then we see that 

\begin{equation*}
\mathcal{L} \circ Cusp_{(4, 5; -7)}(i s) = (s^4, s^5 + s^7) = Cusp_{(4, 5; +7)}(s). 
\end{equation*}

\end{remark}

\section{\bf Examples}\label{sec: example of (4,5)-cusp}

\subsection{Classifications for $(4,5)$-cusps}

\begin{itemize}
\item[(1)] $Cusp_{(4, 5)}(t)$, $p_{1}(t) = (t^4 + t^7, t^5 )$ and $p_{2}(t) = (t^4 - t^7, t^5 )$ are $\mathcal{A}$-$equivalent$ at $t=0$ by Theorem \ref{thm: complete classifications of (4, 5)cusps } $(1)$.

\item[(2)] $Cusp_{(4, 5; +7)}(t)$, $q_{1}(t) =  (t^4 + t^7, t^5 + t^7)$, $q_{2}(t) =  (t^4 - t^7, t^5 + t^7)$ and $q_{3}(t) =  (t^4 - t^6, t^5)$ are $\mathcal{A}$-$equivalent$ at $t=0$ by Theorem \ref{thm: complete classifications of (4, 5)cusps } $(2)$.

\item[(3)] $Cusp_{(4, 5; -7)}(t)$, $c_{1}(t) = (t^4 + t^7, t^5 - t^7)$, $c_{2}(t) = (t^4 - t^7, t^5 - t^7)$, $c_{3}(t) = (t^4, t^5 + t^6)$, $c_{4}(t) = (t^4, t^5 - t^6)$, and $c_{5}(t) = (t^4 + t^6, t^5)$
are $\mathcal{A}$-$equivalent$ at $t=0$ by Theorem \ref{thm: complete classifications of (4, 5)cusps } $(3)$.
\end{itemize}

\subsection{Evolutes of the $(4,5)$-cusp}

Let $\gamma(t) = (t^4, t^5)$ be the $(4,5)$-cusp. Since $\nu(t) = \dfrac{1}{\sqrt{16 + 25t^2}} (-5t, 4)$ and $\mu(t) = \dfrac{-1}{\sqrt{16 + 25t^2}} (4, 5t)$, we have 

\begin{align}
l(t) = \dfrac{20}{16 + 25t^2} \quad \beta(t) = -t^3 \sqrt{16 + 25t^2}. \label{align:curvarture001}
\end{align}
By using (\ref{align:the n-th evolutes of fronts}) and (\ref{align:curvarture001}), we get the representations of the evolute, the second evolute and the third evolute of $\gamma$ as follows.

\begin{align*}
Ev(\gamma)(t) &= \Bigl( -3 t^4 - \dfrac{25}{4} t^6, \dfrac{16}{5}t^3 + 6 t^5 \Bigl), \\
Ev(Ev(\gamma))(t) &= \Bigl( - \dfrac{192}{25}t^2 - 39 t^4 - \dfrac{175}{4}t^6, -\dfrac{32}{5}t^3 - 39 t^5 - \dfrac{375}{8}t^7 \Bigl), \\
Ev(Ev(Ev(\gamma)))(t) &= \Bigl(  \dfrac{192}{25}t^2 + 141 t^4 + \dfrac{925}{2}t^6 + \dfrac{13125}{32}t^8, -\dfrac{1536}{125}t -\dfrac{752}{5}t^3 -444 t^5 - 375 t^7 \Bigl). 
\end{align*}

\begin{figure}[htbp]
\begin{center}
\begin{tabular}{cc}
\begin{minipage}{0.3\hsize}
\begin{center}
\includegraphics[clip,scale=0.35,bb=0 0 467 465]{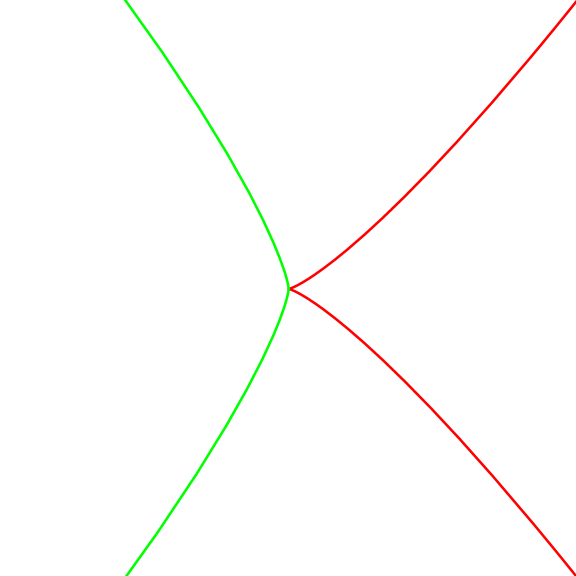}
\caption{}
The red curve is the $(4,5)$ cusp and the green curve is $Ev(\gamma)(t)$.
\end{center}
\end{minipage}

\begin{minipage}{0.3\hsize}
\begin{center}
\includegraphics[clip,scale=0.35,bb=0 0 323 459]{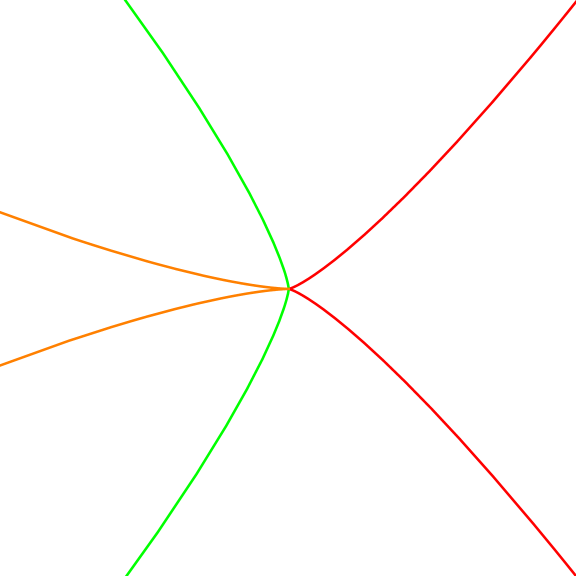}
\caption{}
The orange curve is $Ev(Ev(\gamma))$.
\end{center}
\end{minipage}

\begin{minipage}{0.3\hsize}
\begin{center}
\includegraphics[clip,scale=0.34,bb=0 0 323 459]{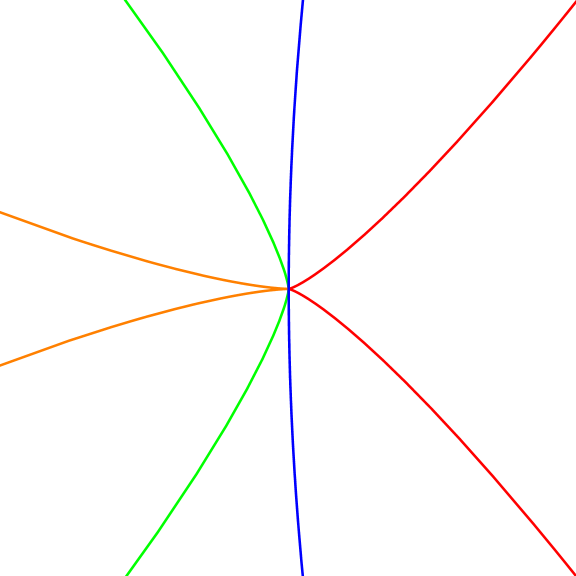}
\caption{}
The blue curve is $Ev(Ev(Ev(\gamma)))$. 
\end{center}
\end{minipage}

\end{tabular}
\end{center}
\end{figure} 

\subsection*{Acknowledgements}
The author would like to express sincere gratitude to Professors Miyuki Koiso and Atsufumi Honda for fruitful discussions and comments.

\end{document}